\documentclass[lean]{afm}
\usepackage{csquotes}
\usepackage{thmtools}

\usepackage{enumitem}

\newcommand{\createtheoremtype}[3]{%
  \newtheorem{#1}[theoremcounter]{\MakeUppercase #2}%
  \newlist{#1parts}{enumerate}{1}%
  \setlist[#1parts]{label=(\alph*),ref=\thetheoremcounter~(\alph*)}%
  \crefname{#1}{\MakeUppercase #2}{\MakeUppercase #3}%
  \Crefname{#1}{\MakeUppercase #2}{\MakeUppercase #3}%
  \crefalias{#1partsi}{#1}%
}
\createtheoremtype{constructionhelper}{construction}{constructions}

\theoremstyle{remark}
\createtheoremtype{fact}{fact}{facts}
\theoremstyle{theorem}

\newlist{tfae}{enumerate}{1}
\setlist[tfae]{label=(\roman*)}
\crefname{tfae}{condition}{conditions}

\usepackage{tikz-cd}
\tikzset{
  symbol/.style={
    draw=none,
    every to/.append style={
      edge node={node [sloped, transform shape, auto=false]{$#1$}}}
  }
}

\newcommand{\N}{\mathbb N}

\newcommand{\R}{\mathbb R}

\DeclareMathOperator{\dom}{dom}

\newcommand{\indf}[2][]{%
  \mathrel{\mathop{\vcenter{
        \hbox{\oalign{\noalign{\kern-.3ex}\hfil$\vert$\hfil\cr
              \noalign{\kern-.7ex}
              $\smile$\cr\noalign{\kern-.3ex}}}
      }}\displaylimits_{#2}^{#1}}}

\newcommand{\set}[1]{\left\{#1\right\}}
\newcommand{\setb}[2]{\set{#1\,\middle|\,#2}}

\newcommand{\abs}[1]{\left|#1\right|}

\newcommand{\mscr}{\mathscr}
\newcommand{\xto}{\xrightarrow}

\renewcommand{\phi}{\varphi}
\renewcommand{\theta}{\vartheta}

\newcommand{\permutationhead}[2]{\ifnum #1=#2 #1\else{#1}&\permutationhead{\the\numexpr #1+1}{#2}\fi}

\definecolor{keywordcolor}{rgb}{0.7, 0.1, 0.1}   
\definecolor{commentcolor}{rgb}{0.4, 0.4, 0.4}   
\definecolor{symbolcolor}{rgb}{0.0, 0.1, 0.6}    
\definecolor{sortcolor}{rgb}{0.1, 0.5, 0.1}      
\definecolor{errorcolor}{rgb}{1, 0, 0}           
\definecolor{stringcolor}{rgb}{0.5, 0.3, 0.2}    
\definecolor{tacticcolor}{rgb}{0, 0, 0}
\usepackage{listings}

\usepackage{xspace}
\usepackage{fontawesome} 
\newcommand{\link}{\xspace\ensuremath{{}^\text{\faExternalLink}}}

\title[A formalization of Borel determinacy in Lean]{A formalization of Borel determinacy in Lean}

\author[Sven Manthe]{Sven Manthe}

\authorinfo[Sven Manthe]{University of Bonn}{sven.manthe@uni-bonn.de}

\VOLUME{2}
\YEAR{2026}
\NUMBER{2}
\firstpage{38}
\DOI{https://doi.org/10.46298/afm.15202}
\receiveddate{February 6, 2025}
\finaldate{January 19, 2026}
\accepteddate{February 9, 2026}


\begin{abstract}
  We present a formalization of Borel determinacy in the Lean 4 theorem prover. The formalization includes a definition of Gale-Stewart games and a proof of Martin's theorem stating that Borel games are determined. The proof closely follows Martin's ``A purely inductive proof of Borel determinacy''.
\end{abstract}

\support{I was supported by the Hausdorff Center for Mathematics at the University of Bonn (DFG EXC 2047). I thank Floris van Doorn for constant support in learning Lean, fixing problems in this project, and writing this paper.}

\begin{document}

\section{Introduction}
Determinacy of infinite two-player games, introduced by Gale and Stewart \cite{galestewart}, is a central topic in descriptive set theory. While determinacy for large classes of sets is intimately connected with large cardinals, Borel determinacy is provable in Zermelo-Fraenkel set theory with the axiom of choice, ZFC \cite{boreldet}.\par
We formalize the proof in the Lean 4 theorem prover \cite{leansystem}, also using its mathematical library, mathlib \cite{mathlib}. Lean is an interactive theorem prover based on the Calculus of Inductive Constructions. Its theory is slightly stronger than ZFC---precisely, it is equiconsistent with $\text{ZFC}\cup\setb{\text{there are }n\text{ inaccessible cardinals}}{n\in\N}$ \cite{typetheorylean} and proves the same statements of higher order arithmetic by the construction provided therein. Thus, Borel determinacy, which is famous for requiring a much larger fragment of ZFC than most common theorems \cite{boreldetrepl}, still holds in Lean's type theory by the same proof as in ZFC. In contrast, stronger forms of determinacy, like analytic determinacy, are not provable in Lean without additional axioms (unless Lean is inconsistent).\par
The formalization follows the presentation in \cite{indboreldet}, with some adaptations. The proof proceeds by induction on the construction of Borel sets, establishing not determinacy directly, but the stronger property of unravelability (about which not much is known beyond the Borel sets \cite{unravunknown}). A key step is the unravelability of closed games, where we follow the proof closely. The handling of countable unions of games, however, required changes. The proof involves an inverse limit of games, which stabilizes levelwise. All stabilized levels are identified, which is problematic to formalize directly. To overcome this, we use a category theoretic lemma. Furthermore, we replace the use of transfinite induction to show unravelability with a direct induction on the construction of Borel sets. This requires a strengthening of unravelability, which we term ``universal unravelability''.\footnote{The account in \cite{moschovakisdst} provides another approach using induction on Borel sets, defining sets that ``unravel fully''.}\par
Our formalization deviates from the mathlib convention of using ``junk values'' instead of propositional hypotheses when implementing partial functions. We compare the advantages of the two approaches. We also briefly discuss the current limitations of Lean's support for propositional hypotheses and our workarounds.\par
For no class beyond the closed sets has determinacy been formalized in a proof assistant before. The project\footnote{\url{https://github.com/sven-manthe/A-formalization-of-Borel-determinacy-in-Lean}} consists of about 5,000 lines of code. The foundational part, including the definitions of Gale-Stewart games and determinacy, comprises less than half of the code. The majority of the code formalizes Martin's proof itself.\par
The article is organized as follows. We first provide the statement of Borel determinacy and sketch Martin's proof in \cref{sec:outline}, providing the necessary context for the subsequent formalization details. In \cref{sec:details} we turn to a technical description of issues specific to the formalization of Borel determinacy. Finally, we discuss the more general design choice of implementing partial functions via propositional hypotheses instead of ``junk values'' in \cref{sec:partial}. We provide static links to formalized lemmas with the \href{https://archive.softwareheritage.org/swh:1:dir:7228df9d31b4c2da3f3de47daff98f18cd966f1d;origin=https://github.com/sven-manthe/A-formalization-of-Borel-determinacy-in-Lean;visit=swh:1:snp:6860cd60adfdddb6a8bec689ddb852388d32bd5d;anchor=swh:1:rev:267694b2241ce1900123d6d9fa15bb4a0b42a603}{\link} symbol.
\section{Outline of the informal proof}\label{sec:outline}
In this section, we define Gale-Stewart games and sketch a minor variation of Martin's proof of Borel determinacy. For a more comprehensive treatment, see \cite[Section 20]{kechrisdst}. Several other proofs of Borel determinacy can be found in the literature \cites{boreldet}{indboreldet}[Theorem 20.5]{kechrisdst}[Theorem 6F.1]{moschovakisdst}[403-470]{analyticconf}{studentboreldet}.
\subsection{Notation and definitions}
Our notation is mostly set theoretic, not type theoretic. We regard $0\in\N$. The prefix relation on finite sequences is denoted by $\preceq$. The length of a finite sequence $x$ is denoted by $\abs x$. For a function $f$ whose domain consists of finite sequences, we write $f|_{=n},f|_{\le n}$ for the restrictions of $f$ to $\setb{x\in\dom f}{\abs x=n}$, respectively $\setb{x\in\dom f}{\abs x\le n}$.\par
A \textbf{tree} on a set $A$ is a set $T$ of finite sequences in $A$ such that $x\preceq y\in T$ implies $x\in T$. A tree is \textbf{pruned} if it has no leaves, that is, $\preceq$-maximal elements. Its \textbf{body} $[T]$ is the set of infinite branches, that is, elements of $A^\N$ all whose finite prefixes lie in $T$.\par
A \textbf{Gale-Stewart game} is a pair $G=(T,P)$, where $T$ is a nonempty pruned tree and $P\subseteq[T]$ is the payoff set. The game is played between two players, $0$ and $1$, choosing alternately elements from $A$ such that the resulting finite sequence lies in $T$, starting with player $0$. Their infinite play results in a branch $a\in[T]$. Player $0$ wins if $a\in P$; otherwise, player $1$ wins. A \textbf{strategy} $\sigma$ of player $i\in\set{0,1}$ is a map $\sigma$ associating to each position $x$ of player $i$ (that is, $x\in T$ with $\abs x\equiv i\mod2$) a successor position. A play is \textbf{consistent} with $\sigma$ if throughout the play, each move of player $i$ is given by $\sigma$. A strategy $\sigma$ yields a tree of partial plays compatible with $\sigma$, and we identify $\sigma$ with this tree. The strategy $\sigma$ is \textbf{winning} if all plays consistent with $\sigma$ are won by player $i$, that is, $[\sigma]\subseteq P$ if $i=0$ and $[\sigma]\subseteq[T]\setminus P$ if $i=1$.\par
It is straightforward to show that at most one player has a winning strategy. A game is \textbf{determined} if at least one player has a winning strategy. There are undetermined games even for $A=\set{0,1}$, as can be seen by counting strategies and possible payoff sets or by employing a nonprincipal ultrafilter. (Without the axiom of choice, there are still undetermined games on the power set $2^\R$ and $\aleph_1$, but determinacy for $\N$ or $\R$ is related to large cardinals.)\par
The Borel $\sigma$-algebra of a topological space is the $\sigma$-algebra generated by the open sets. For a tree $T$ on $A$, we equip the body $[T]\subseteq A^\N$ with the subspace topology of the product topology of the discrete topologies on $A$. A game is Borel (respectively open, closed, clopen...) if its payoff set is so as a subset of $[T]$.
\subsection{Proof sketch}
We can now state Borel determinacy and sketch a proof.
\begin{theorem}\cite{boreldet}\href{https://archive.softwareheritage.org/swh:1:cnt:7b42f5e3634650a4f2c474ce13dd0be9be2de27b;origin=https://github.com/sven-manthe/A-formalization-of-Borel-determinacy-in-Lean;visit=swh:1:snp:6860cd60adfdddb6a8bec689ddb852388d32bd5d;anchor=swh:1:rev:267694b2241ce1900123d6d9fa15bb4a0b42a603;path=/BorelDet/Proof/borel_determinacy.lean;lines=158}{\link}
  Every Borel game is determined.
\end{theorem}
The proof for closed games is essentially the proof of Zermelo's theorem, that is, determinacy of finite games. Namely, if player $1$ has no winning strategy, then player $0$ has a strategy that never moves into a winning position for player $1$. If the game is closed, then any such strategy of player $0$ is winning. In particular, closed (and open) games are determined, but the more precise version in the previous sentence will be used throughout the proof.\par
Since finite unions of determined sets are not necessarily determined, the Borel induction needs to pass through a strengthened concept, universal unravelability, which indeed yields a $\sigma$-algebra.\par
Let $\mscr C$ be the category whose objects are pairs $(A,T)$ where $T$ is a tree on $A$ and whose morphisms $(A,S)\to(B,T)$ are isotone, length preserving maps $S\to T$. The body map $T\mapsto[T]$ extends to a functor from $\mscr C$ to topological spaces\href{https://archive.softwareheritage.org/swh:1:cnt:e98cd599715e6fcc5698f02e594ed00e8cf405fb;origin=https://github.com/sven-manthe/A-formalization-of-Borel-determinacy-in-Lean;visit=swh:1:snp:6860cd60adfdddb6a8bec689ddb852388d32bd5d;anchor=swh:1:rev:267694b2241ce1900123d6d9fa15bb4a0b42a603;path=/BorelDet/Tree/body_functor.lean;lines=90}{\link}, and we denote its action on morphisms also by $f\mapsto[f]$. A \textbf{covering} $T'\to T$ is a pair of a morphism $\pi\colon T'\to T$ in $\mscr C$ and a map $\phi$ from strategies in $T'$ to strategies in $T$ of the same player such that $\sigma|_{\le n}=\sigma'|_{\le n}$ implies $\phi(\sigma)|_{\le n}=\phi(\sigma')|_{\le n}$ and the following lifting condition holds: For $x\in[\phi(\sigma)]$ there is $x'\in[\sigma]$ with $[\pi](x')=x$. A covering $(\pi,\phi)$ is $k$-\textbf{fixing}, or a $k$-covering, if the restriction of $\pi$ to the first $k$ levels is a bijection and $\phi$ is the map induced by $\pi$ on these levels.\par
A game $(T,P)$ is \textbf{unravelable} if for all $k\in\N$ there is a $k$-covering $(\pi,\phi)\colon T'\to T$ such that $[\pi]^{-1}(P)$ is clopen\href{https://archive.softwareheritage.org/swh:1:cnt:e086164446706288737a98f3441a0230e33e93d2;origin=https://github.com/sven-manthe/A-formalization-of-Borel-determinacy-in-Lean;visit=swh:1:snp:6860cd60adfdddb6a8bec689ddb852388d32bd5d;anchor=swh:1:rev:267694b2241ce1900123d6d9fa15bb4a0b42a603;path=/BorelDet/Proof/covering.lean;lines=164}{\link}. $(T,P)$ is \textbf{universally unravelable} if for all coverings $(\pi,\phi)\colon T'\to T$, the game $(T',[\pi]^{-1}(P))$ is unravelable\href{https://archive.softwareheritage.org/swh:1:cnt:7b42f5e3634650a4f2c474ce13dd0be9be2de27b;origin=https://github.com/sven-manthe/A-formalization-of-Borel-determinacy-in-Lean;visit=swh:1:snp:6860cd60adfdddb6a8bec689ddb852388d32bd5d;anchor=swh:1:rev:267694b2241ce1900123d6d9fa15bb4a0b42a603;path=/BorelDet/Proof/borel_determinacy.lean;lines=46}{\link}. It is easy to show that a covering $(\pi,\phi)\colon T'\to T$ maps winning strategies in $(T',[\pi]^{-1}(P))$ to winning strategies in $(T,P)$. Thus the determinacy of clopen games implies the determinacy of unravelable games \cite[Lemma 1]{indboreldet}\href{https://archive.softwareheritage.org/swh:1:cnt:e086164446706288737a98f3441a0230e33e93d2;origin=https://github.com/sven-manthe/A-formalization-of-Borel-determinacy-in-Lean;visit=swh:1:snp:6860cd60adfdddb6a8bec689ddb852388d32bd5d;anchor=swh:1:rev:267694b2241ce1900123d6d9fa15bb4a0b42a603;path=/BorelDet/Proof/covering.lean;lines=166}{\link}. Then Borel determinacy follows from the following two lemmas:
\begin{lemma}\label{unravsig}\href{https://archive.softwareheritage.org/swh:1:cnt:7b42f5e3634650a4f2c474ce13dd0be9be2de27b;origin=https://github.com/sven-manthe/A-formalization-of-Borel-determinacy-in-Lean;visit=swh:1:snp:6860cd60adfdddb6a8bec689ddb852388d32bd5d;anchor=swh:1:rev:267694b2241ce1900123d6d9fa15bb4a0b42a603;path=/BorelDet/Proof/borel_determinacy.lean;lines=120}{\link}
  For a fixed tree $T$, the set of $P$ such that $(T,P)$ is universally unravelable forms a $\sigma$-algebra.
\end{lemma}
\begin{lemma}\cite[Lemma 3]{indboreldet}\label{unravclo}\href{https://archive.softwareheritage.org/swh:1:cnt:7b42f5e3634650a4f2c474ce13dd0be9be2de27b;origin=https://github.com/sven-manthe/A-formalization-of-Borel-determinacy-in-Lean;visit=swh:1:snp:6860cd60adfdddb6a8bec689ddb852388d32bd5d;anchor=swh:1:rev:267694b2241ce1900123d6d9fa15bb4a0b42a603;path=/BorelDet/Proof/borel_determinacy.lean;lines=55}{\link}
  Closed games are universally unravelable.
\end{lemma}
\begin{proof}[Proof sketch of \cref{unravsig}]
  Stability under complements is trivial.\href{https://archive.softwareheritage.org/swh:1:cnt:7b42f5e3634650a4f2c474ce13dd0be9be2de27b;origin=https://github.com/sven-manthe/A-formalization-of-Borel-determinacy-in-Lean;visit=swh:1:snp:6860cd60adfdddb6a8bec689ddb852388d32bd5d;anchor=swh:1:rev:267694b2241ce1900123d6d9fa15bb4a0b42a603;path=/BorelDet/Proof/borel_determinacy.lean;lines=48}{\link} For countable unions, we employ an inverse limit construction, where we use the parameter $k$ in $k$-coverings.\par
  Specifically, given a diagram
  \[\dots\to T_3\xto{(\pi_2,\phi_2)}\to T_2\xto{(\pi_1,\phi_1)}T_1\xto{(\pi_0,\phi_0)}T_0\]
  such that each $(\pi_i,\phi_i)$ is a $(k+i)$-covering, the projections of its limit $T=\lim_iT_i$ in $\mscr C$ can be extended to $(k+i)$-coverings. The main idea is to define the strategies levelwise by first mapping to a sufficiently large $T_i$ and then applying the transition maps. The lifts required in the definition of a covering can be constructed levelwise similarly. Since this is precisely \cite[Lemma 4]{indboreldet}, we omit the details.\par
  To deduce \cref{unravsig}, we use a variant of the argument of \cite[Theorem]{indboreldet}. Namely, given a sequence $P_i$ such that the $(T,P_i)$ are universally unravelable, recursively define an inverse system of trees such that each transition map unravels an additional $[\pi_i\circ\dots\circ\pi_0]^{-1}(P_i)$. The limit of this system is a game whose payoff set is the preimage of $\bigcup_iP_i$ under the first projection. The limit game is open, thus universally unravelable\href{https://archive.softwareheritage.org/swh:1:cnt:7b42f5e3634650a4f2c474ce13dd0be9be2de27b;origin=https://github.com/sven-manthe/A-formalization-of-Borel-determinacy-in-Lean;visit=swh:1:snp:6860cd60adfdddb6a8bec689ddb852388d32bd5d;anchor=swh:1:rev:267694b2241ce1900123d6d9fa15bb4a0b42a603;path=/BorelDet/Proof/borel_determinacy.lean;lines=58}{\link} by \cref{unravclo}. The composite of the covering unraveling the limit with the projection witnesses unravelability of $(T,\bigcup_iP_i)$.
\end{proof}

\begin{proof}[Proof sketch of \cref{unravclo}]
  We first observe that it suffices to show that closed games are unravelable since $[\pi]^{-1}(P)$ is closed for a covering $(\pi,\phi)\colon T'\to T$ of a closed game $(T,P)$. This is also the statement of \cite[Lemma 3]{indboreldet}\href{https://archive.softwareheritage.org/swh:1:cnt:7b42f5e3634650a4f2c474ce13dd0be9be2de27b;origin=https://github.com/sven-manthe/A-formalization-of-Borel-determinacy-in-Lean;visit=swh:1:snp:6860cd60adfdddb6a8bec689ddb852388d32bd5d;anchor=swh:1:rev:267694b2241ce1900123d6d9fa15bb4a0b42a603;path=/BorelDet/Proof/borel_determinacy.lean;lines=26}{\link}.\par
  To simplify notation, we only construct a $0$-covering $(A',T')\to(A,T)$. In the first move, player $0$ plays not only a move $a_0\in A$, but also a quasistrategy (a strategy possibly allowing several moves in a given position) $\sigma$ of herself. In the second move, player $1$ responds either with a finite sequence $x$ compatible with $\sigma$ such that she wins every play extending $x$ or with a quasistrategy of herself that guarantees losing against $\sigma$. In either case, the players follow the played strategies (until the finite sequence is finished in the first case). There is a morphism in $\mscr C$ from this game to the original game given by forgetting the played strategies. Moreover, the winner is determined after the first moves of both players, whence the game is clopen. We omit the description of the strategy part of the covering, which actually is the most involved part of the proof of Borel determinacy.
\end{proof}
Let us also note where the proof theoretic strength of Borel determinacy appears in this proof. Namely, starting with a game $((A,T),P)$, the construction of \cref{unravclo} yields a game on a set of cardinality $2^{\#A}$. Thus, the construction of \cref{unravsig} leads to an underlying set of cardinality $\beth_\omega(\#A)$. Finally, this construction needs to be iterated inside of a Borel induction, leading to sets of cardinality $\beth_\alpha(\#A)$ for arbitrarily large countable ordinals $\alpha$.
\subsection{Comparison to Martin's proof}
We now discuss the main differences between our presentation (and formalization) and \cite{indboreldet}. While Martin does not explicitly use the category $\mscr C$, the limit construction is the same. However, he does not only require $k$-coverings to be bijective on the first $k$ levels, but to equal the identity there; mathematically, this does not matter.\par
The definition of universally unravelable is our own: Martin only uses unravelability. Thus, he does not show that the unravelable sets form a $\sigma$-algebra, but instead proves by transfinite induction on the levels of the Borel hierarchy that Borel sets are unravelable. Our modification enables the recursive construction of the inverse system in the argument for countable unions, thus avoiding the explicit use of ordinals in Martin's induction. Proving universal unravelability instead of unravelability is easy in all required cases. It is not clear to the author whether the unravelable sets also yield a $\sigma$-algebra or even whether every unravelable set is universally unravelable. The Borel sets satisfy a further strengthening of universal unravelability, where preimages under coverings are replaced by preimages under arbitrary continuous maps from trees, defined as ``unraveling fully'' in \cite[before 6F.3]{moschovakisdst} for the same purpose.
\section{Formalizing the proof}\label{sec:details}
We now describe some technical aspects of the formalization of the variant of Martin's proof outlined in \cref{sec:outline}.
\subsection{Gale-Stewart games}
We begin by discussing design choices in the formalization of Gale-Stewart games. The formalization of closed determinacy\href{https://archive.softwareheritage.org/swh:1:cnt:96b3bc9798948c2ec543dbd8190bdbfb4511b615;origin=https://github.com/sven-manthe/A-formalization-of-Borel-determinacy-in-Lean;visit=swh:1:snp:6860cd60adfdddb6a8bec689ddb852388d32bd5d;anchor=swh:1:rev:267694b2241ce1900123d6d9fa15bb4a0b42a603;path=/BorelDet/Game/gale_stewart.lean;lines=78}{\link} is relatively straightforward and thus we will not discuss it.\par
We represented finite sequences as lists, which have a developed API in mathlib. The existence of Cantor subsets of sufficiently definable, e.g., determined, sets, is usually proved by combining finite approximations into a Cantor scheme \cite[Section 6.A]{kechrisdst}. In contrast to the construction of Cantor schemes in mathlib\href{https://github.com/leanprover-community/mathlib4/blob/06bd3096751d751e3267390c3eed35f3a49d8a8c/Mathlib/Topology/MetricSpace/CantorScheme.lean#L51}{\link}, which also relies on lists to approximate branches of the tree $2^\N$, we did not reverse lists when compared to standard mathematical notation. For example, the first move of a game corresponds to the first element of the list, not the last one.\par
The definitions of trees\href{https://github.com/leanprover-community/mathlib4/blob/06bd3096751d751e3267390c3eed35f3a49d8a8c/Mathlib/SetTheory/Descriptive/Tree.lean#L25}{\link} and games\href{https://archive.softwareheritage.org/swh:1:cnt:00952ab227ef5ce78a836a035669677c8bab797a;origin=https://github.com/sven-manthe/A-formalization-of-Borel-determinacy-in-Lean;visit=swh:1:snp:6860cd60adfdddb6a8bec689ddb852388d32bd5d;anchor=swh:1:rev:267694b2241ce1900123d6d9fa15bb4a0b42a603;path=/BorelDet/Game/games.lean;lines=12}{\link} based on this list representation were straightforward. However, the formalization of strategies required some choices. First, in addition to the notions of strategy\href{https://archive.softwareheritage.org/swh:1:cnt:d38cdffd5509761b07b53a564a62b93a815caa53;origin=https://github.com/sven-manthe/A-formalization-of-Borel-determinacy-in-Lean;visit=swh:1:snp:6860cd60adfdddb6a8bec689ddb852388d32bd5d;anchor=swh:1:rev:267694b2241ce1900123d6d9fa15bb4a0b42a603;path=/BorelDet/Game/strategies.lean;lines=117}{\link} and quasistrategy\href{https://archive.softwareheritage.org/swh:1:cnt:d38cdffd5509761b07b53a564a62b93a815caa53;origin=https://github.com/sven-manthe/A-formalization-of-Borel-determinacy-in-Lean;visit=swh:1:snp:6860cd60adfdddb6a8bec689ddb852388d32bd5d;anchor=swh:1:rev:267694b2241ce1900123d6d9fa15bb4a0b42a603;path=/BorelDet/Game/strategies.lean;lines=109}{\link}, we proved most lemmas in terms of the weaker notion of \textbf{prestrategy}\href{https://archive.softwareheritage.org/swh:1:cnt:d38cdffd5509761b07b53a564a62b93a815caa53;origin=https://github.com/sven-manthe/A-formalization-of-Borel-determinacy-in-Lean;visit=swh:1:snp:6860cd60adfdddb6a8bec689ddb852388d32bd5d;anchor=swh:1:rev:267694b2241ce1900123d6d9fa15bb4a0b42a603;path=/BorelDet/Game/strategies.lean;lines=10}{\link}. While a quasistrategy allows to have several possible moves, a prestrategy additionally allows to have no possible moves in a position. In other words, a prestrategy is just a function from positions to sets of moves. In particular, in contrast to quasistrategies, the existence of winning prestrategies is trivial: the empty prestrategy is vacuously winning.\par
The following remarks apply mutatis mutandis to all of three notions of strategy. Some textbooks and \cite{indboreldet} define a (quasi-)strategy not as a map sending a position to a successor position, called \textbf{functional strategies} in the following, but rather as the tree of positions compatible with the strategy, called \textbf{strategy trees} here. While there is a canonical surjection from functional strategies onto strategy trees, this map is not injective (namely, a functional strategy also specifies moves in positions that could only be reached by violating the strategy before). However, due to the surjection, which also preserves the property of a strategy to be winning, this difference is irrelevant as far as only the question which player wins a game is concerned. In particular, it is irrelevant in the study of determinacy, the primary application of Gale-Stewart games. While strategy trees have a more reasonable notion of equality, equality of strategies is rarely relevant in applications or proofs of determinacy.\par
While formalizing, functional strategies turned out to be more convenient to work with than strategy trees:
\begin{lstlisting}
  def PreStrategy := ∀ x : T, IsPosition x.val p → Set (ExtensionsAt x)
\end{lstlisting}
The main reason is that many specific functional strategies are easier to formally define than the associated strategy trees. Defining a strategy tree often requires recursion that can be avoided when defining a corresponding functional strategy. Moreover, one would need to prove ``for positions of the opposing player, all possible successors lie in the tree'' whenever defining a strategy tree directly.\par
In contrast, strategy trees are sometimes easier to use, for example in defining the notion of a winning strategy\href{https://archive.softwareheritage.org/swh:1:cnt:00952ab227ef5ce78a836a035669677c8bab797a;origin=https://github.com/sven-manthe/A-formalization-of-Borel-determinacy-in-Lean;visit=swh:1:snp:6860cd60adfdddb6a8bec689ddb852388d32bd5d;anchor=swh:1:rev:267694b2241ce1900123d6d9fa15bb4a0b42a603;path=/BorelDet/Game/games.lean;lines=63}{\link}. We thus defined the canonical surjection\href{https://archive.softwareheritage.org/swh:1:cnt:d38cdffd5509761b07b53a564a62b93a815caa53;origin=https://github.com/sven-manthe/A-formalization-of-Borel-determinacy-in-Lean;visit=swh:1:snp:6860cd60adfdddb6a8bec689ddb852388d32bd5d;anchor=swh:1:rev:267694b2241ce1900123d6d9fa15bb4a0b42a603;path=/BorelDet/Game/strategies.lean;lines=23}{\link} and used it to define such notions, but did not formulate the more explicit conditions on a tree to be a strategy tree. More precisely, we did not formalize the lemma stating ``a tree lies in the image of the canonical surjection mentioned above if and only if for positions of the opposing player, all possible successors lie in the tree.''\par
Nonetheless, functional strategies were disadvantageous in some situations. For instance, the requirement in the definition of a $k$-covering $(\pi,\phi)$ that $\phi$ be the map induced by $\pi$ on the first $k$ levels is redundant when working with strategy trees. In that setting, it follows from the lifting condition (and is thus not included in Martin's definition). However, this argument fails for functional strategies since we need not find a lift for positions incompatible with the strategy. We added the additional condition and maintained it through the formalization to avoid the case distinctions that would otherwise be necessary.
\subsection{Categories of trees}
Here we discuss the formalization of the inverse limit construction, which is needed to handle countable unions.\par
Variants of Martin's definition of a $k$-covering \cites[20.C]{kechrisdst}[before 6F.3]{moschovakisdst} require the restriction to the first $k$ levels not only to be bijective, but to be the identity, and Martin assumes this in the proof of \cite[Lemma 4]{indboreldet}. While this might be cumbersome already in a purely set theoretical formalization, it presents additional problems in a type theoretic setting. Thus, we chose to drop this requirement and formulate the necessary commutative diagrams involving these restrictions and their inverses instead.\par
Martin's construction of the inverse limit of the $T_i$ was restricted to the case where the $k$-th transition map $T_{k+1}\to T_k$ is $k$-fixing. In this case, the limit tree is levelwise in bijection with some $T_i$. We opted for a less ad-hoc approach to these limits. Namely, the category $\mscr C$, as defined earlier, has all limits\href{https://archive.softwareheritage.org/swh:1:cnt:6afb2f62f2bdc0bf82f55eaada5901a98ad00c0f;origin=https://github.com/sven-manthe/A-formalization-of-Borel-determinacy-in-Lean;visit=swh:1:snp:6860cd60adfdddb6a8bec689ddb852388d32bd5d;anchor=swh:1:rev:267694b2241ce1900123d6d9fa15bb4a0b42a603;path=/BorelDet/Tree/tree_lim.lean;lines=136}{\link}. From general facts on limits of cofiltered systems\href{https://github.com/leanprover-community/mathlib4/blob/06bd3096751d751e3267390c3eed35f3a49d8a8c/Mathlib/CategoryTheory/Limits/Constructions/EventuallyConstant.lean#L125}{\link}, we deduced that the $k$-th projection is $k$-fixing in the case of fixing transition maps\href{https://archive.softwareheritage.org/swh:1:cnt:6afb2f62f2bdc0bf82f55eaada5901a98ad00c0f;origin=https://github.com/sven-manthe/A-formalization-of-Borel-determinacy-in-Lean;visit=swh:1:snp:6860cd60adfdddb6a8bec689ddb852388d32bd5d;anchor=swh:1:rev:267694b2241ce1900123d6d9fa15bb4a0b42a603;path=/BorelDet/Tree/tree_lim.lean;lines=150}{\link}.\par
We formalized the category $\mscr C$\href{https://archive.softwareheritage.org/swh:1:cnt:5f7f3651d4b78259514e0d7631aab767f312626e;origin=https://github.com/sven-manthe/A-formalization-of-Borel-determinacy-in-Lean;visit=swh:1:snp:6860cd60adfdddb6a8bec689ddb852388d32bd5d;anchor=swh:1:rev:267694b2241ce1900123d6d9fa15bb4a0b42a603;path=/BorelDet/Tree/len_tree_hom.lean;lines=22}{\link} using mathlib's category theory library. In fact, $\mscr C$ is a topos. It is equivalent to the category of sheaves on the total order $(\N,\le)$ regarded as category, where a tree $T$ corresponds to the sheaf $F$ whose sections at $n\in\N$ are the nodes at the $n$-th level of $T$, and where $\lim F=[T]$. However, we did not formalize this result, but proved directly that $\mscr C$ has all limits. We then proved that the functor $E_n$ given by restricting to the $n$-th level preserves limits\href{https://archive.softwareheritage.org/swh:1:cnt:6afb2f62f2bdc0bf82f55eaada5901a98ad00c0f;origin=https://github.com/sven-manthe/A-formalization-of-Borel-determinacy-in-Lean;visit=swh:1:snp:6860cd60adfdddb6a8bec689ddb852388d32bd5d;anchor=swh:1:rev:267694b2241ce1900123d6d9fa15bb4a0b42a603;path=/BorelDet/Tree/tree_lim.lean;lines=101}{\link} by constructing a left adjoint\href{https://archive.softwareheritage.org/swh:1:cnt:6afb2f62f2bdc0bf82f55eaada5901a98ad00c0f;origin=https://github.com/sven-manthe/A-formalization-of-Borel-determinacy-in-Lean;visit=swh:1:snp:6860cd60adfdddb6a8bec689ddb852388d32bd5d;anchor=swh:1:rev:267694b2241ce1900123d6d9fa15bb4a0b42a603;path=/BorelDet/Tree/tree_lim.lean;lines=76}{\link} (retrospectively, it would likely have been easier to prove the preservation of limits by a direct argument, which would have been equivalent by a adjoint functor theorem). Applying the standard computation of cofiltered limits of systems of isomorphisms to the image of a final segment of the original diagram under $E_k$ then yields the desired result that the $k$-th projection is $k$-fixing.\par
Afterwards, we used the category theory library mostly as a notational tool to manage compositions of morphisms. For instance, we defined categories with morphisms given by maps of strategies\href{https://archive.softwareheritage.org/swh:1:cnt:e086164446706288737a98f3441a0230e33e93d2;origin=https://github.com/sven-manthe/A-formalization-of-Borel-determinacy-in-Lean;visit=swh:1:snp:6860cd60adfdddb6a8bec689ddb852388d32bd5d;anchor=swh:1:rev:267694b2241ce1900123d6d9fa15bb4a0b42a603;path=/BorelDet/Proof/covering.lean;lines=46}{\link} or coverings\href{https://archive.softwareheritage.org/swh:1:cnt:e086164446706288737a98f3441a0230e33e93d2;origin=https://github.com/sven-manthe/A-formalization-of-Borel-determinacy-in-Lean;visit=swh:1:snp:6860cd60adfdddb6a8bec689ddb852388d32bd5d;anchor=swh:1:rev:267694b2241ce1900123d6d9fa15bb4a0b42a603;path=/BorelDet/Proof/covering.lean;lines=102}{\link}, without employing further theorems of category theory. The remainder of the formalization of the inverse limit argument\href{https://archive.softwareheritage.org/swh:1:cnt:a7038a7402b9f3863056f4b460dc50f3e729fb7e;origin=https://github.com/sven-manthe/A-formalization-of-Borel-determinacy-in-Lean;visit=swh:1:snp:6860cd60adfdddb6a8bec689ddb852388d32bd5d;anchor=swh:1:rev:267694b2241ce1900123d6d9fa15bb4a0b42a603;path=/BorelDet/Proof/covering_lim.lean;lines=234}{\link} was tedious but straightforward, following \cite{indboreldet} closely. However, the identifications made along the $k$-fixing transition maps needed to be translated back into commutativity of certain diagrams.
\subsection{Unraveling closed sets}
This subsection discusses the final part of the formalization, the unravelability of closed games. In this part, the categorical machinery developed for the limit argument was not required.\par
The primary challenge was translating the informal definition of the unraveling game to facilitate the construction of the map of strategies and the proof of the lifting property. The main difficulty stemmed from the fact that, while the unraveling game $G'$ proceeds as the unraveled game $G$ in almost all moves, two moves involve additional data, namely strategies of players and finite sequences. A naïve approach is to formalize the type of moves in $G'$ by a sum type of the moves in $G$ and \lstinline{Option} types specifying the notions of strategies either player may use (say, encoded as functional strategies), but this proved inconvenient in practice. Note that using strategy trees instead of functional strategies as moves would suffice. Furthermore, this choice enables uniform treatment of all three possible types of strategies that are played. Thus, we represented moves in $G'$ by the cartesian product of the moves $A$ in $G$ and trees on $A$\href{https://archive.softwareheritage.org/swh:1:cnt:d2c3876734fd8a04d9b455d5e0e7c9bf0e5cd7de;origin=https://github.com/sven-manthe/A-formalization-of-Borel-determinacy-in-Lean;visit=swh:1:snp:6860cd60adfdddb6a8bec689ddb852388d32bd5d;anchor=swh:1:rev:267694b2241ce1900123d6d9fa15bb4a0b42a603;path=/BorelDet/Proof/covering_closed_game.lean;lines=16}{\link}, with the second component encoding which moves are allowed according to the implicitly played strategies:\\
\begin{minipage}{\linewidth}
\begin{lstlisting}
  def A' := A × Tree A
  ...
  def ValidExt (x : List A') (a : A') := [a.1] ∈ getTree x ∧
    if x.length = 2 * k then
      ∃ S : QuasiStrategy (subAt (getTree x) [a.1]) Player.one, a.2 = S.1.subtree
    else if h : x.length = 2 * k + 1 then
      LosingCondition (x ++ [a]) (by simpa) ∨ WinningCondition (x ++ [a]) (by simpa)
    else a.2 = subAt (getTree x) [a.1]
  def gameTree : Tree A' where
    val := {x | List.reverseRecOn x True (fun x a hx ↦ hx ∧ ValidExt x a)}
\end{lstlisting}
\end{minipage}
The game rules then enforce that the trees played are induced by some such strategy\href{https://archive.softwareheritage.org/swh:1:cnt:d2c3876734fd8a04d9b455d5e0e7c9bf0e5cd7de;origin=https://github.com/sven-manthe/A-formalization-of-Borel-determinacy-in-Lean;visit=swh:1:snp:6860cd60adfdddb6a8bec689ddb852388d32bd5d;anchor=swh:1:rev:267694b2241ce1900123d6d9fa15bb4a0b42a603;path=/BorelDet/Proof/covering_closed_game.lean;lines=82}{\link}.
\subsection{Automation}
Maintaining the assumptions ``$f$ is $k$-fixing'' and ``$x$ is a position of player $i$'' throughout the argument, although trivially omitted on paper, presented technical challenges.\par
The requirements ``$x$ is a position of player $i$,'' after unfolding some length assumptions, often reduced to formulas of Presburger arithmetic (asking whether $\abs x$, which often was a sum of lengths of subsequences, is even). We developed a tactic\href{https://archive.softwareheritage.org/swh:1:cnt:ec64f3a31aa1f85ce770a2a6df6c12b7981eccc8;origin=https://github.com/sven-manthe/A-formalization-of-Borel-determinacy-in-Lean;visit=swh:1:snp:6860cd60adfdddb6a8bec689ddb852388d32bd5d;anchor=swh:1:rev:267694b2241ce1900123d6d9fa15bb4a0b42a603;path=/BorelDet/Game/player.lean;lines=29}{\link} that does the necessary unfolding of list lengths and then invokes Lean's builtin partial Presburger solver \lstinline{omega}.\par
For the assumptions ``$f$ is $k$-fixing'', we employed a similar approach with an additional refinement. The typical use case involved expressions $f^{-1}x$. In this case, Lean should deduce that $f$ is $\abs x$-fixing from the context and use this information to interpret the expression. To a term $f$, we could usually associate a maximal $k\in\N$ syntactically such that $f$ is $k$-fixing (to be precise, $k$ is maximal such that there is an ``obvious'' proof that $f$ is $k$-fixing, but $f$ may be $(k+1)$-fixing for other reasons). Our solution leveraged Lean's typeclass inference mechanism, using this maximal $k$ as a ``out-parameter'' of a typeclass\href{https://archive.softwareheritage.org/swh:1:cnt:2ff192c2caec1e1bf562330fb286b7a0cc8a1a81;origin=https://github.com/sven-manthe/A-formalization-of-Borel-determinacy-in-Lean;visit=swh:1:snp:6860cd60adfdddb6a8bec689ddb852388d32bd5d;anchor=swh:1:rev:267694b2241ce1900123d6d9fa15bb4a0b42a603;path=/BorelDet/Tree/restrict_tree.lean;lines=88}{\link}:
\begin{lstlisting}
  class Fixing (k : outParam ℕ) (f : S ⟶ T) : Prop where prop : IsIso ((res k).map f)
\end{lstlisting}
Then Lean automatically infers a suitable $k$. To finish, we apply a procedure similar as before\href{https://archive.softwareheritage.org/swh:1:cnt:2ff192c2caec1e1bf562330fb286b7a0cc8a1a81;origin=https://github.com/sven-manthe/A-formalization-of-Borel-determinacy-in-Lean;visit=swh:1:snp:6860cd60adfdddb6a8bec689ddb852388d32bd5d;anchor=swh:1:rev:267694b2241ce1900123d6d9fa15bb4a0b42a603;path=/BorelDet/Tree/restrict_tree.lean;lines=120}{\link} to prove $\abs x\le k$. This tactic was used as a default argument to the inverse function.
\section{Dependent types and ``junk'' values}\label{sec:partial}
\subsection{General considerations}
When formalizing partial functions $f\colon X\to Y$, there are two primary design choices. For the first one, let $P$ be a predicate defining the domain $\dom f$ and represent $f$ as a function $\setb{x\in X}{P(x)}\to Y$. For the second one, express $f$ as a total function $X\to Y\amalg\set{\bot}$, where $X\setminus f^{-1}(\set\bot)=\dom f$. Several minor variants are possible. In the first case, one could use the curried version $(x\colon X)\to P(x)\to Y$ instead. In the second case, often $\bot$ is replaced by a carefully chosen element of $Y$, e.g. 0, avoiding the need to extend the codomain---in this case, one still defines $P$, which enables to recover the domain.\par
In informal mathematics, the first variant is generally preferred. A notable exception is the statement of Fubini's theorem on product measures, $\int f(x,y)d(x,y)=\int(\int f(x,y)dy)dx$. Here the inner integral $\int f(x,y)dy$ only exists for almost all $x$, and one usually declares it to be $0$ otherwise to define the iterated integral. Some abuse of notation, like writing $\lim_{n\to\infty}\frac{x^n}{n-1}$ although the quotient is not defined in all cases, can also be more easily justified using the second approach. Lean's mathlib largely adopts the second approach, with some exceptions, such as the indexing of lists and arrays. It also uses the first approach when choosing an element of $Y$ instead of adding a new one is not possible, for example in the type theoretic axiom of choice. Our project, however, prefers the first approach, for example in the definition of strategies, which are undefined in positions of the opponent. In the following, we compare the two approaches.\par
Let us first consider advantages of the first approach. If the necessary hypotheses $P(x)$ are included in the term as an argument, then syntactical rewriting operations, replacing $x$ by a propositionally equal $x'$, can automatically infer $P(x')$. With the second approach, after a series of rewriting operations, it may become necessary to perform corresponding rewriting operations separately on $P(x)$. When such partial functions are nested, this usually increases the number of required rewriting operations linearly. For example, if $(f_n\circ\dots\circ f_1)(x)$ is rewritten to $(f_n\circ\dots\circ f_1)(x')$, then one also needs to rewrite at all hypotheses stating that $(f_i\circ\dots\circ f_1)(x')$ lies in the intended domain of $f_n\circ\dots\circ f_{i+1}$. This is particularly problematic if these hypotheses are provided not as explicit terms but by automation, which may succeed on input $(f_i\circ\dots\circ f_1)(x)$, but fail on $(f_i\circ\dots\circ f_1)(x')$. Another, albeit minor, advantage of the first approach is that the contextual information provided by the hypotheses $P(x)$ in the term itself can make certain lemmas applicable, thus enhancing tactics like Lean's \lstinline{simp}.\par
Next we describe advantages of the second approach. While the first approach yields subterms with complex dependent types, the second approach avoids the use of dependent types. This is often beneficial since many standard lemmas, like congruence lemmas for functions, are more easily stated in a non-dependently typed context. It also simplifies rewriting along propositional equalities of subterms since replacing a term by an equal term otherwise may lead to ill-typed expressions (this issue appears if the first variant is used in the uncurried version). It also becomes possible to input partial functions into higher-order functions: for example, one may take limits of functions defined in a neighborhood or integrals of functions defined almost everywhere. Another minor advantage arises if $\bot$ is replaced by a carefully chosen element of $Y$: some identities may remain valid even without the hypothesis $P(x)$, whence some proofs can omit this hypothesis.\par
We finally justify the choice of the first approach in our formalization with an example. The definition
\begin{lstlisting}
  def PreStrategy := ∀ x : T, IsPosition x.val p → Set (ExtensionsAt x)
\end{lstlisting}
uses the dependence on \lstinline{IsPosition} instead of junk values. Thus, when rewriting $x$ to $y$ in an expression involving applications of strategies to complex terms, the property \lstinline{IsPosition y.val p} was maintained automatically. Otherwise, after such steps it would often be necessary to unfold $y$ to reprove that it is a position of $p$, that is, determine the parity of $\abs y$.
\subsection{Lean and dependent types}
Finally, we discuss some challenges encountered during the formalization due to our choice of the first approach, which is not well-supported by some of Lean's tactics due to the dominance of the second approach in most Lean projects, including mathlib. We expect that these problems could be fixed by improving support for dependent types in Lean.\par
For rewriting, Lean primarily provides two tactics, \lstinline{rewrite[lemma]} and \lstinline{simp only[lemma]}. A main difference is that \lstinline{rewrite} only uses a single lemma, while \lstinline{simp only} is obtained by specializing a tactic \lstinline{simp} that employs ``discrimination trees'' to attempt rewriting with many different lemmas. Moreover, \lstinline{rewrite} traverses terms top down, while \lstinline{simp} traverses terms bottom up. To us, more subtle differences involving the implementation are more relevant.\par
The \lstinline{rewrite} tactic is more primitive. It directly applies path recursion, that is, the type theoretic recursion lemmas for the inductive type of equality. A ``motive'' for this recursion lemma (roughly, a function $f$ such that for $x=y$ the term $fx$ is rewritten to $fy$) is created by traversing the term. Thus, for example, it is impossible to rewrite subterms containing bound variables. The problem most relevant to us is that these motives need not be type correct if dependent types are involved, and \lstinline{rewrite} fails in such cases. Thus \lstinline{rewrite} is mainly usable with the second approach to partial functions, although it is more powerful in edge cases. After our formalization, a variant \lstinline{rewrite!} for dependent types was added to Lean.\par
The \lstinline{simp} tactic does not use path induction directly, but instead generates congruence lemmas for the substitutions. In particular, it can handle dependent types by adjusting propositional hypotheses $P(x)$ whose types are affected by these replacements along $x=x'$ to $P(x')$. (Changing the proof term in this and similar cases is irrelevant since propositional extensionality is built-in to Lean as a definitional equality.) This enables the advantages of the first approach.\par
However, also \lstinline{simp} sometimes struggled with dependent types. Consider, for example, a function $f\colon(x\colon X)\to Y(x)$ for a type family $Y$, and a second function $g\colon(x\colon X)\to Y(x)\to Z$. They can be composed to a non-dependent function $gf\colon X\to Z$ by $gf(x)=g(x)(f(x))$. While rewriting $f(x)$ to $f(x')$ is problematic as it changes the type (and is thus not performed by \lstinline{simp}), rewriting $gf(x)$ to $gf(x')$ is not. However, \lstinline{simp} does not recognize this unless a manual congruence lemma is added for the specific pair of functions $f$ and $g$. In the project, we added such congruence lemmas manually for all necessary pairs of functions (e.g., \lstinline{Subtype.val} and application of strategies \href{https://archive.softwareheritage.org/swh:1:cnt:d38cdffd5509761b07b53a564a62b93a815caa53;origin=https://github.com/sven-manthe/A-formalization-of-Borel-determinacy-in-Lean;visit=swh:1:snp:6860cd60adfdddb6a8bec689ddb852388d32bd5d;anchor=swh:1:rev:267694b2241ce1900123d6d9fa15bb4a0b42a603;path=/BorelDet/Game/strategies.lean;lines=121}{\link}), but this process could be automated. Similarly, the discrimination trees used in the implementation of \lstinline{simp} to avoid considering unnecessary subterms check for equality only up to a certain level of reducibility. This mechanism failed in cases where unfolding reducible definitions was required from the lemma being applied to the rewritten statement, instead of vice versa. We addressed this issue by annotating some arguments with \lstinline{no_index}, which excludes them from consideration in the discrimination tree matching.\par
The final issue concerns performance issues arising from unnecessary nesting and traversal of proof terms. The prevalent approach in Lean and mathlib is to disregard the size of proof terms since they are typically hidden behind lemma names and never need to be unfolded due to propositional extensionality. For instance, when a definition is added to the environment, all its propositional subterms are abstracted as new lemmas and replaced by references to these lemmas. This process avoids traversing the proof term when unfolding the definition afterwards.\par
Unfortunately, this is applied only to the value of a declaration, but not to its type (for example, it is not applied to the statements of lemmas). Moreover, it is performed only after the entire definition has been processed. Thus, if later parts of the expression depend on earlier parts, then tactics evaluated to construct the later parts may traverse proof terms for the earlier parts.\par
Neither restriction is problematic if the second approach to partial functions is used exclusively (in fact, in this case even the implemented post-processing step could be omitted). Hence, this usually causes no performance issues in mathlib (the few cases where propositional arguments are used either do not involve excessively large proof terms, like most list indexing, or are handled carefully by the end user, like the axiom of choice). Constant factor increases in runtime, caused for example as the partial Presburger solver \lstinline{omega} produces large proof terms, can already be noticeable. However, when using the first approach to partial functions and nesting them, then proof terms may be substituted into proof terms repeatedly. This leads to an increase of the running time exponential in the number of dependent hypotheses. For example, the running time of the following code\footnote{Kim Morrison minimized this at \url{https://github.com/leanprover/lean4/issues/5108}} is exponential in $n$:\\
\begin{minipage}{\linewidth}
\begin{lstlisting}
  #time
  example (A : Type) (p : Bool) : ([] : List A).length % 1 = 0 := by
    iterate n let h1 : ([] : List A).length % 1 = 0 := by cases p <;> sorry
    sorry
\end{lstlisting}
\end{minipage}
The obvious solution to this performance problem is to eagerly abstract proof terms as lemmas. Performing this abstraction manually is not feasible. While it is already cumbersome to replace proof terms like the ones for ``$f$ is $\abs x$-fixing'' implicitly constructed in $f^{-1}x$ by hand, the more serious problem is that the effort to do so scales quadratically in the size of the proof code. Namely, when abstracting a proof as a lemma, one must not only state the often short proof term (for example, \lstinline{by simp; omega}), but also the type of the expression, which otherwise could be deduced from the context. The size of this type alone grows linearly with the code, causing the quadratic scaling. To address this, we wrote a metaprogram \lstinline{abstract} that performs the abstraction eagerly, deducing the type from the context:\\
\begin{minipage}{\linewidth}
\begin{lstlisting}
  elab "abstract" tacs:ppDedent(tacticSeq) : tactic => do
    let target ← getMainTarget
    let goal ← getMainGoal
    let newGoal ← mkFreshExprMVar target
    setGoals [newGoal.mvarId!]
    Elab.Tactic.evalTactic tacs
    setGoals [goal]
    goal.assign (← mkAuxTheorem ((← getDeclName?).get! ++ `abstract ++ (← mkFreshId)) target newGoal)
\end{lstlisting}
\end{minipage}
Replacing problematic occurrences of \lstinline{by tactics} by \lstinline{by abstract tactics} resolved the performance issues. Independently, a metaprogram \lstinline{as_aux_lemma} with the same functionality was added to the core language\footnote{\url{https://github.com/leanprover/lean4/pull/6823}} for the implementation of tree maps. We finally used \lstinline{as_aux_lemma} instead of \lstinline{abstract} in the repository.
\section{Related work}
\cite{galestewartisabelle} formalized the determinacy of closed games, one of our intermediate results, in Isabelle. They used coinductive lists to handle finite and infinite plays simultaneously, while we restricted attention to infinite plays and use finite sequences only to describe partial plays.\par
Mathlib includes a formalization of finite games, \lstinline{SetTheory.Game}, following \cite{conwaygames}. This formalism only handles finite games and includes determinacy, which always holds in this context, implicitly.
\section{Conclusions}
The formalization followed \cite{indboreldet} closely, only requiring expansions of the argument for parts that are inherently informal, like some descriptions of strategies or the identification of sets standing in canonical bijection. It also demonstrates that using consistency-wise strong fragments of ZFC does not pose an obstruction to formalizing in itself.\par
An interesting direction for future work would be to formalize determinacy for larger classes of sets under large cardinal hypotheses, and possibly also the converse statements constructing inner models with large cardinals assuming determinacy.\par
While basic results of descriptive set theory have already been formalized in mathlib, several consequences of Borel determinacy that also could be proven independently are still missing. Thus, Borel determinacy could also serve as a tool for building a more comprehensive library of formalized descriptive set theory.\par
In the author's opinion, mathlib's choice to prefer encoding partial functions using ``junk'' values instead of mirroring informal mathematical definitions is not necessary and may even be a disadvantage. It would be interesting to see other formalization projects encode partial functions by restricting the domain, particularly after support for this in Lean improved.
\printbibliography
\end{document}